\documentclass[11pt]{amsart}

 \usepackage{amsmath,amssymb}
 \usepackage{color}
 \usepackage{mathrsfs}
 \usepackage{graphicx}
 
   \newcommand{\zit}[1]{(\ref{#1})}
     \def\dint{\int\hspace{-5pt}\int}

\usepackage[colorlinks, pagebackref, linkcolor=red, citecolor=blue]{hyperref}
 
 \usepackage{calc}
\makeatletter

\newlength{\temp@wc@width}

\newlength{\temp@wc@height}

\newcommand{\widecheck}[1]{%
\setlength{\temp@wc@width}{\widthof{$#1$}}%
\setlength{\temp@wc@height}{\heightof{$#1$}}%
#1\hspace{-\temp@wc@width}%
\raisebox{\temp@wc@height+2pt}[\heightof{$\widehat{#1}$}]%
{\rotatebox[origin=c]{180}{\vbox to 0pt{\hbox{$\widehat{\hphantom{#1}}$}}}}%
}

\makeatother
 
 \numberwithin{equation}{section}
 \newtheorem{theorem}{Theorem}[section]
 \newtheorem{lemma}[theorem]{Lemma}
 
 \newtheorem{remark}[theorem]{Remark}
 \newtheorem{example}[theorem]{Example}
 \newtheorem{definition}[theorem] {Definition}

 \newtheorem{corollary}[theorem]{Corollary}
 \newtheorem{proposition}[theorem]{Proposition}
 \theoremstyle{definition}
 \newtheorem{exemple}[theorem]{Example}

 \def\dist{\operatorname{dist}}

 \def\T{ \mathbb T}
 \def\R{ \mathbb R}

 \def\D{{ \mathbb D}}
  
 \def\C{{ \mathbb C}}
 \def\Z{{ \mathbb Z}}
 \def\N{{ \mathbb N}}

 \def\e{\varepsilon}
 \def\dbar{\ov\partial}
 
 \def\sp {\quad}
 \def\ssc{\scriptscriptstyle}

 \def\union{\cup}
 \def\Union{\bigcup}
 \def\inter{\cap}
 
 \def\ov{\overline}

 \def\ss{\subseteq}
 \def\emp{\emptyset}

 \def\buildrel#1_#2^#3{\mathrel{\mathop{\kern 0pt#1}\limits_{#2}^{#3}}}
 
 \def\ssi{\Longleftrightarrow}
 \def\imp{\Longrightarrow}

\begin{document}
  
  \title [Algebra generators]{Holomorphic injective extensions of functions in $P(K)$ 
 and algebra generators}

  \author{Raymond Mortini}
  
\address{
Universit\'{e} de Lorraine\\
 D\'{e}partement de Math\'{e}matiques et Institut \'Elie Cartan de Lorraine,  UMR 7502\\
 Ile du Saulcy\\
 F-57045 Metz, France} 
 \email{raymond.mortini@univ-lorraine.fr}


  \begin{abstract}
 We present  necessary and sufficient conditions on planar compacta $K$ and  continuous
 functions $f$ on $K$ in order that $f$ generates the algebras $P(K), R(K), A(K)$ or $C(K)$.
 We also unveil quite surprisingly simple examples  of non-polynomial convex 
 compacta $K\ss\C$ and $f\in P(K)$
 with the property that $f \in P(K)$ is a homeomorphism, but for which $f^{-1}\notin P(f(K))$.
 As a consequence, such functions do not admit injective holomorphic  extensions to the
 interior of the  polynomial convex hull $\widehat K$.
On the other hand,  it will be shown that the restriction $f^*|_G$ of the  Gelfand-transform $f^*$
 of an injective function $f\in P(K)$
{\bf is} injective on every regular, bounded  complementary component  $G$ of $K$.
 A necessary and sufficient condition in terms of the behaviour of
$f$ on the outer boundary of $K$  is given in order $f$  admits a holomorphic  injective extension
to $\widehat K$.
We also include some results on the existence of continuous logarithms on punctured compacta 
containing the origin in their boundary.
  \end{abstract}

   \maketitle

 \centerline {\small\the\day.\the \month.\the\year} \medskip
 
 \section*{Introduction}
 
 Let $K$ be  a compact set in the complex plane $\C$. As usual, $P(K)$ denotes
 the set of complex-valued continuous functions on $K$ that can  be uniformly approximated 
 by polynomials. Endowed with the usual algebraic operations and the supremum norm, $P(K)$
 is a uniformly closed subalgebra of $C(K)$. By definition, the monomial $z$ is a generator for
 $P(K)$. We recall  the following definition:

  \begin{definition}
 If $A$ is a commutative  unital Banach algebra and $S$ a subset of $A$,
then the smallest closed subalgebra of $A$ containing $S$ is denoted by $[S]_{{\rm alg}}$.
We also say that $[S]_{{\rm alg}}$ is the algebra generated by $S$. 
\end{definition}

Note that $[S]_{{\rm alg}}$ is the norm-closure of the set of all
polynomials of the form $\sum a_\iota f_1^{n_1}\dots f_j^{n_j}$, where  $f_k\in S$, 
$\iota=(n_1,\dots,n_j)\in \N^j$ and $j\in \N^*$.

We are interested in the following question:  which functions  are generators for $P(K)$?
We also consider the associated algebras 
$$\mbox{$A(K)=\{f\in C(K): f $ holomorphic in the interior $K^\circ$ of $K\}$},$$
and $R(K)$, the uniform closure of the set  $R_0(K)$ of rational functions without poles on $K$.

We  present  in Section 1, which represents the motivational part of this paper,
 the answer to this question.  
 The description in the case of the algebra $P(K)$ leads to the following
problem: if $f\in P(K)$ is a homeomorphism, is the unique holomorphic extension $f^*$ 
of $f$ to the polynomial convex hull  $\widehat K$ of $K$  injective?

In the case where $K$ is the unit circle $\T$, a classical result, known under the name
of the Darboux-Picard theorem (see \cite[p.~310]{bu})  
tells us that $f^*$ actually is injective on the closed unit disk $\mathbb D$.  Generalizations
in various directions had been established  (see \cite{bu}).  The 
general situation, however, does not seem to have been solved. 
We give a nice  example showing that the answer to the preceding question is negative.
Our main goal  then will be  achieved in Section 2, namely a proof of  the following  result:
if $f\in P(K)$ then the Gelfand transform, $f^*$, of $f$  is injective 
on $\widehat K$ if and only if $f$ maps
 the outer boundary of $K$  onto the outer boundary of $f(K)$.
Our method involves Eilenberg's  representation theorem  for zero-free functions
on compacta as well as a homotopic variant of Rouch\'e's theorem. 
As a corollary we obtain that  for {\it every} injective function  $f\in P(K)$, 
the restriction  $f^*|_G$ of $f^*$ to a  regular hole $G$ of $K$ is  injective.
Here a hole $G$ of $K$ is called {\it regular}  if $G$ is the only hole of its boundary.
In particular, if $K$ has a connected complement and a connected interior, then
$f^*$ is injective on $K$ if and only if  $f\in P(\partial K)$ is injective.

In Section 3 we deal with a feature not covered by Eilenberg's theorem: under which conditions
on $K$ with $0\in \partial K$ does there exist a continuous branch of the logarithm on $K\setminus \{0\}$? (In Eilenberg's theorem $0$ belongs to the complement of $K$). 

\section{Algebra Generators}

  \begin{theorem}\label{ck}
 Let $K\ss\C$ be compact and $\varphi\in C(K,\C)$. The following assertions are equivalent:
 \begin{enumerate}
 \item [(1)] $\varphi$ is a generator for $C(K,\C)$; that is $C(K,\C)=[\varphi]_{{\rm alg}}$;
 \item[(2)] $\varphi$ is a homeomorphism  of $K$ onto $\varphi(K)$, $K^ \circ=\emp$ and
 $\C\setminus K$ is connected.
\end{enumerate}
\end{theorem}
\begin{proof}
It is clear that every generator for  $C(K,\C)$ is point separating. Hence,
$\varphi$ must be a homeomorphism of $K$ onto its image.
Let $f\in C(K,\C)$. We first show that $f\in [\varphi]_{{\rm alg}}$ if and only if
$f\circ \varphi^{-1}\in P(\varphi(K))$.  In fact,
$f\in [\varphi]_{{\rm alg}}$ if and only if $p_n(\varphi)\to f$ uniformly on $K$ for some sequence of polynomials
 $p_n\in \C[z]$. But
 $$\max_{z\in K}|p_n(\varphi(z))-f(z)|\to 0\ssi 
 \max_{w\in \varphi (K)} |p_n(w)-f(\varphi^{-1}(w))|\to 0.$$

This in turn is equivalent to $f\circ \varphi^{-1}\in P(\varphi(K))$. 
Next we observe that every $h\in C(\varphi(K),\C)$ writes as $f\circ \varphi^{-1}$ for some
$f\in C(K,\C)$; just put $f=h\circ\varphi$. We conclude that the assumption
 $C(K,\C)=[\varphi]_{{\rm alg}}$ is equivalent to the assumption
 $C(\varphi (K),\C)=P(\varphi(K))$, whenever $\varphi$ is an homeomorphism.
  By Lavrentiev's theorem \cite[p. 192]{b},
 this happens if and only if $\varphi (K) ^\circ=\emp$ and $\C\setminus\varphi (K)$ is connected.
 Now  $\varphi  (K) ^\circ=\emp$ if and only if $K^\circ=\emp$. Moreover,  the number of connected components of the complement of
 a compact set in $\C$ is invariant under homeomorphisms (see \cite[p. 99]{bu}).
  Hence condition (2) is
 necessary and sufficient for $C(K,\C)$ to be singly generated by $\varphi$.
\end{proof}

 \begin{remark}
  Let  $K\ss\R$ be compact and $\varphi\in C(K,\R)$. The following assertions are equivalent:
 \begin{enumerate}
 \item [(1)] $\varphi$ is a generator for $C(K,\R)$; that is $C(K,\R)=[\varphi]_{{\rm alg}}$;
 \item[(2)] $\varphi$ is a homeomorphism  of $K$ onto $\varphi(K)$. 
  \end{enumerate}
\end{remark}
\begin{proof}
As above, if $\varphi$ is a homeomorphism of $K$ onto its image,  the assumption
 $C(K,\R)=[\varphi]_{{\rm alg}}$ is equivalent to the assumption 
 $C(\varphi (K),\R)=P_\R(\varphi(K))$. \footnote{ This is, per definition, the uniform closure
 of the set of {\it real} polynomials on $\varphi(K)$.}
  This is always true, though, by Weierstrass' approximation theorem.
\end{proof}

 \begin{theorem}\label{ak}
 Let $K\ss\C$ be compact and $\varphi\in A(K)$. The following assertions are equivalent:
 \begin{enumerate}
 \item [(1)] $\varphi$ is a generator for $A(K)$; that is $A(K)=[\varphi]_{{\rm alg}}$;
 \item[(2)] $\varphi$ is a homeomorphism  of $K$ onto $\varphi(K)$ and
 $\C\setminus K$ is connected.
\end{enumerate}
\end{theorem}
\begin{proof}
As in the previous theorem, we obtain that the assumption
 $A(K)=[\varphi]_{{\rm alg}}$ is equivalent to the assumption 
 $A(\varphi (K))=P(\varphi(K))$ whenever $\varphi\in A(K)$ is an homeomorphism.
 Note that $\varphi^{-1}\in A(\varphi(K))$.
 By Mergelyan's theorem \cite{rud1},
 this happens if and only if  $\C\setminus K$ is connected.
 \end{proof}
 
 The proof  of the corresponding  result for $R(K)$ and $P(K)$ needs an additional argument:
 
 \begin{lemma}\label{ratio}
 Let $K\ss\C$ be compact and $\varphi\in C(K)$. The following assertions hold: 
 \begin{enumerate}
 \item [(1)]
If $\varphi\in R(K)$, then $h\in R(\varphi(K))$ implies that $f:=h\circ \varphi\in R(K)$.
\item [(2)]  
If $\varphi\in P(K)$, then $h\in P(\varphi(K))$ implies that $f:=h\circ \varphi\in P(K)$.
 \end{enumerate}
\end{lemma}
\begin{proof}
(1)  Let $(r_n(w))$ denote a sequence of rational functions without poles on $\varphi(K)$
converging uniformly on  $\varphi(K)$ to $h(w)$. Then 
\begin{equation}\label{erste}
\max_{z\in K}|r_n(\varphi(z))-h(\varphi(z))|\to 0.
\end{equation}
Next, let  $(\varphi_n(z))$ be  a sequence of rational functions without poles on $K$ 
converging uniformly on $K$ to $\varphi(z)$. We claim that the following assertions hold:
\begin{enumerate}
\item[i)] \label{rat1}
 For every $n$ there exists $j_n>n$ such that $r_n\circ \varphi_{j_n}$ is a rational function 
without poles on $K$.
\item[ii)] \label{rat2}
$(r_n\circ\varphi_{j_n})$ converges uniformly on $K$ to  $h\circ\varphi$.
\end{enumerate}
In fact, since it is obvious that $r_n\circ\varphi_j$ is a rational function again,  
it remains to prove for i) that  $j\geq n$ can be chosen so that $r_n\circ\varphi_j$ has no poles on $K$. 
To see this, we observe that $r_n$ has no poles in the closure of an open  
neighborhood $U_n$ of $\varphi(K)$. 
Let $\e_n=\dist(\varphi(K), \C\setminus U_n)$. The compactness of $\varphi(K)$
 implies that $\e_n>0$.
Since $||\varphi_j-\varphi||_K\to 0$, $\dist(\varphi_j(z), \varphi(K))<\e_n/2$ for every $z\in K$ 
and $j\geq j^*_n>n$. Thus, for all $z\in K$ and  $j\geq j^*_n$,  $\varphi_{j^*_n}(z)\in U_n$.
 Hence $r_n\circ \varphi_{j}$
has no poles on $K$ when $j\geq j^*_n$.  This gives i).

ii)  Fix $n$. Since $r_n$ is uniformly continuous on  $\ov U_n$, we may 
choose $j_n\geq j^*_n$ so big that
\begin{equation}\label{zweite}
||r_n\circ \varphi_{j_n}-r_n\circ\varphi||_K< 1/n.
\end{equation}

Then ii)  is a consequence of  the following  estimations:
\begin{eqnarray*}
|r_n\circ \varphi_{j_n}-h\circ \varphi|&\leq & |r_n\circ\varphi_{j_n}- r_n\circ \varphi|
+|r_n\circ \varphi- h\circ \varphi|\\
&\buildrel\leq_{\zit{erste}}^{\zit{zweite}} & 1/n+\e/2<\e
\end{eqnarray*}
for $n\geq n_0$.
We conclude that $h\circ\varphi\in R(K)$.\\
(2) This works as in part ii) above, where rational functions are replaced by polynomials.
Note that i)  is irrelevant here.
\end{proof}

\begin{theorem}
 Let $K\ss\C$ be compact and $\varphi\in R(K)$. The following assertions are equivalent:
 \begin{enumerate}
 \item [(1)] $\varphi$ is a generator for $R(K)$; that is $R(K)=[\varphi]_{{\rm alg}}$;
 \item[(2)] $\varphi$ is a homeomorphism  of $K$ onto $\varphi(K)$ and
 $\C\setminus K$ is connected.
\end{enumerate}
\end{theorem}
\begin{proof}
As usual, we see that for homeomorphic maps $\varphi$ and 
$f\in R(K)$  one has $f\in [\varphi]_{{\rm alg}}$ if and only if
$f\circ \varphi^{-1}\in P(\varphi(K))$. 

(1) $\imp$ (2)~  Let $h\in R(\varphi(K))$. Since, by assumption, $\varphi\in R(K)$, we deduce 
from Lemma \ref{ratio} that $f:=h\circ \varphi\in R(K)$. 
Hence $h=f\circ \varphi^{-1}\in P(\varphi(K))$ if $\varphi$ is a generator for $R(K)$. Thus $P(\varphi(K))=R(\varphi(K))$.
By Runge's theorem, $\varphi(K)$ has connected complement, and so the same is true for $K$.

(2) $\imp$ (1)~ If $K$ (and so $\varphi(K)$),  has connected complement, then by Mergelyan's
Theorem, see  \cite{rud1}, $P(\varphi(K))=R(\varphi(K))=A(\varphi(K))$.   Consider any $f\in R(K)$
and let $h:=f\circ\varphi^{-1}$. Then $h\in A(\varphi(K))$. Hence 
$f\circ \varphi^{-1}=h\in P(\varphi(K))$.
Thus $f\in  [\varphi]_{{\rm alg}}$. Consequently, $R(K)= [\varphi]_{{\rm alg}}$.
\end{proof}

\begin{corollary}
If $A=C(K),  A(K)$ or $R(K)$ is singly generated,  then $K$ is polynomially convex and
 $A=P(K)$.
\end{corollary}
\begin{proof}
This follows from  the previous Theorems which imply that under the given assumption,
$K$  is polynomially convex. Hence, by Mergelyan's Theorem, $P(K)=R(K)=A(K)$,
and in the remaining case, the additional condition $K^\circ=\emp$ implies
that $C(K)=P(K)$.
\end{proof}

\begin{theorem}
 Let $K\ss\C$ be compact and $\varphi\in P(K)$. The following assertions are equivalent:
 \begin{enumerate}
 \item [(1)] $\varphi$ is a generator for $P(K)$; that is $P(K)=[\varphi]_{{\rm alg}}$;
 \item[(2)] $\varphi$ is a homeomorphism  of $K$ onto $\varphi(K)$ and 
 $\varphi^{-1}\in P(\varphi(K))$.

 \end{enumerate}
\end{theorem}
\begin{proof}
(1) $\imp$ (2)~   As usual,  if $\varphi$ is a generator, then $\varphi$ is point separating,
hence a homeomorphism of $K$ onto $\varphi(K)$.   Note also  that for $f\in P(K)$, 
$f\in [\varphi]_{{\rm alg}}$ if and only if $f\circ \varphi^{-1}\in P(\varphi(K))$.
In particular, if $f(z)=z$ then  $\varphi^{-1}\in P(\varphi(K))$. 

(2) $\imp$ (1)~  Let $f\in P(K)$. By Lemma \ref{ratio}\; (2) applied to the inverse function,
 the assumption 
$\varphi^{-1}\in P(\varphi(K))$ implies that  $f\circ \varphi^{-1}\in P(\varphi(K))$.
Hence $f\in [\varphi]_{{\rm alg}}$ and so $P(K)= [\varphi]_{{\rm alg}}$.
\end{proof}

It is now a natural question to ask whether the condition $\varphi^{-1}\in P(\varphi(K))$
is redundant or not?   The following example  shows that it is not.

\begin{exemple}\label{noinjext}
Let $$K=\{z\in\C: |z+1|=1\}\union \{z\in\C: |z-2|=2\}$$ (see figure \ref{inji}).
 \begin{figure}[h] 
   \hspace{2cm}
   \scalebox{.40} {\includegraphics{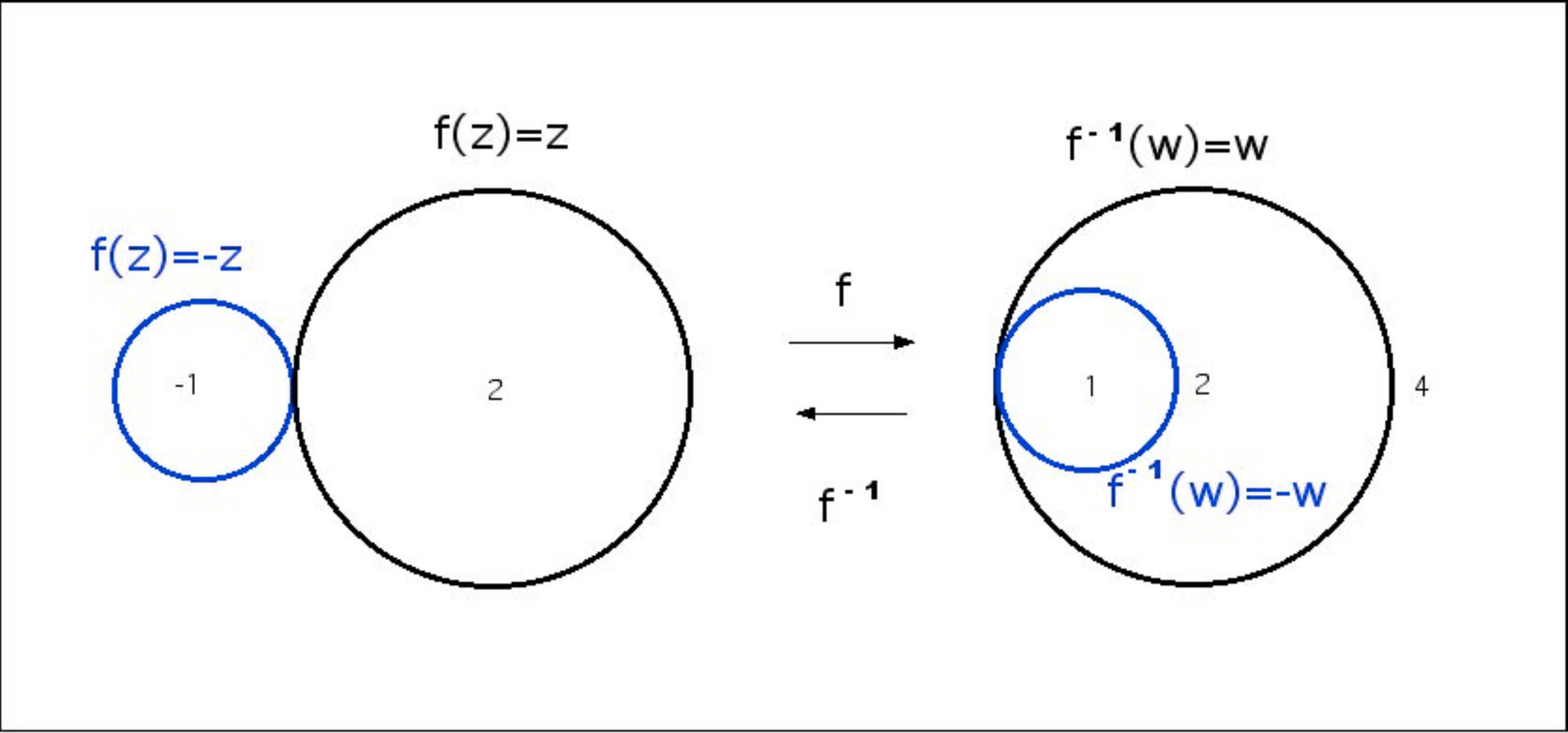}} 
\caption{\label{inji} No injective extension}
\end{figure}

Then the function $f(z)=-z$ for $|z+1|=1$ and $f(z)=z$ for $|z-2|=2$
is injective on $K$ and belongs to $P(K)$, because $f$ has a holomorphic extension
to the polynomial convex hull 
$$\widehat K=\{z\in\C: |z+1|\leq 1\}\union \{z\in\C: |z-2|\leq 2\}$$
of $K$ and so, by Mergelyan's theorem, $f$ can be uniformly approximated on $\widehat K$ by polynomials.

The image $f(K)$ of $K$ under $F$  coincides with the set
$$\{w\in\C: |w-1|=1\}\union \{w\in\C: |w-2|=2\}.$$
Moreover, $f^{-1}(w)= -w$ on $D_1:=\{w\in\C: |w-1|=1\}$ and $f^{-1}(w)=w$ on
$D_2:=\{w\in\C: |w-2|=2\}$.  It is clear that this function does not belong to  $P(f(K))$,
because otherwise, $f^{-1}|_{D_2}$ would have a holomorphic extension to
the polynomial convex hull $\widehat D_2$ of $D_2$. Since this extension can only be $w$  itself,
it  does not coincide with $f^{-1}|_{D_1}(w)=-w$ on $D_1\ss \widehat D_2$.
Note also, that $f$ does not admit a holomorphic  injective extension to $\widehat K$.
\end{exemple}

\begin{proposition}
Let $f\in P(K)$ be a homeomorphism and suppose that $f$ has an injective, holomorphic
extension to the interior of the polynomial convex hull, $\widehat K$,  of $K$.  \footnote{ in the sense
that there is $g\in C(\widehat K)$ such that $g$ is holomorphic in $\widehat K^\circ$ and injective
on $\widehat K$.}  
Then $f^{-1}\in P(f(K))$.

\end{proposition}
\begin{proof}
If $f^*$ denotes this extension, then $f^*$ coincides   with the Gelfand 
transform $\hat f$ of $f$ (in fact, $f^*$ and $\hat f$ belong to $A(\widehat K)$ and $f^*=\hat f=f$
on the Shilov boundary of $A(\widehat K)$, which coincides with $\partial K$). 
Now $(f^*)^{-1}\in A(f^*(\widehat K))$. Since $\widehat K$ has connected complement, 
the invariance theorem \ref{eilenberg}(4) implies that $S:=f^*(\widehat K)$ has connected complement, too. Hence,
by Mergelyan's Theorem, $(f^*)^{-1}\in P(S)$. Restricting to $f(K)\ss S$ yields
that $f^{-1}= (f^*)^{-1}|_{f(K)}\in P(f(K))$, because any sequence of polynomials
converging uniformly on $S$ to $(f^*)^{-1}$ converges a fortiori  uniformly on $f(K)$.
\end{proof}

\section{Injective extensions}
Example \ref{noinjext} shows that $P(K)$-functions which are injective on $K$
do not necessarily have an injective holomorphic extension to the polynomial convex hull
of $K$. A positive result in this direction is known, though:
 
\begin{theorem}[Darboux-Picard]\cite[p.  310]{bu}, \cite{pom}\label{dbp}
Let $f\in A(\D)$ and suppose that $f$ is injective on $\partial\D$. Then
$f$ is injective on $\ov\D$.
\end{theorem}

In the following  we shall deal with the general case of arbitrary compacta. Recall  that  a {\it hole} of a compact set $K$ is a bounded component of $\C\setminus K$ and that the {\it outer boundary},
$S_\infty$, of $K$ is the boundary of the polynomial convex hull $\widehat K$ of $K$. We  need Eilenberg's theorem (see below) and 
the following homotopic variant of Rouch\'e's theorem, the proof of which
is based on  an areal analogue of the argument principle (see \cite[p. 105]{na}).
Here, as usual, 
the  maps $f,g\in C(X,Y)$, defined on Hausdorff spaces $X$ and $Y$, are said to be {\it homotopic} 
in $C(X,Y)$ if there exists a continuous map $H:X\times[0,1]\to Y$
such that $H(x,0)=f(x)$ and $H(x,1)=g(x)$ for every $x\in X$.  

\begin{definition}
For a compact set $K\ss\C$, let  $M(K)$ denote the set of   continuous functions   on $K$ 
that are meromorphic in $K^\circ$.
\end{definition}
Thus, a function in $M(K)$ has only a finite number of poles in $K^\circ$ and none on the boundary.
Of course, $A(K)\ss M(K)$.
Finally, for a  function $f\in M(K)$,
 $n_{\ssc K}(f)$ denotes the number of zeros (possibly infinite) of $f$ in $K^\circ$ and 
 $p_{\ssc K}(f)$ the number of poles of $f$ in $K^\circ$ (including multiplicities).

 \begin{theorem}[{\bf Rouch\'e for homotopic maps}] \label{rouchhomo}
 Let $K\ss\C$ be compact and let $f,g\in M(K)$ be zero-free on $\partial K$.
Suppose that  $f$ and $g$ are homotopic in $C(\partial K, \C^*)$.
Then $n_{\ssc K}(f)-p_{\ssc K}(f)=n_{\ssc K}(g)-p_{\ssc K}(g)$.
\end{theorem}
\begin{proof}
For a proof where  $f$ and $g$ have no poles, that is in the case where $f,g\in A(K)$, we refer to
 \cite{moru}.  Now suppose that  $f,g\in M(K)$.  Since $f$ and $g$ have only a finite number
 of poles and zeros in $K$,   we may write them as
 $$f(z)=\frac{\prod_{j=1}^n (z-a_j)^{n_j}}{\prod_{j=1}^p (z-z_j)^{p_j}}\; \tilde f(z),\sp\sp
g(z)=\frac{\prod_{j=1}^m (z-b_j)^{m_j}}{\prod_{j=1}^q (z-w_j)^{q_j}}\;\tilde g(z),$$
where $\tilde f,\tilde g\in A(K)$ are zero-free and $m_j,n_j, p_j, q_j\in \N^*$. 
Note that a zero of $g$ may be a pole or zero of $f$ and vice versa. Put

 $$h(z):=\prod_{j=1}^p (z-z_j)^{p_j}\prod_{j=1}^q (z-w_j)^{q_j}$$ 
and consider the functions $F:=hf$ and $G:=hg$.

Then $F,G\in A(K)$ and $F$ and $G$ are homotopic in $K(\partial K, \C^*)$ (note that
if $H(z,t)$ is a homotopy between  $f$ and $g$, then 
$$\tilde H(z,t):=h(z)\;H(z,t)$$
is a homotopy in $K(\partial K, \C^*)$ between $F$ and $G$).
Hence, by the homotopic  version of Rouch\'e's theorem for holomorphic functions \cite{moru}, 
 $n_{\ssc K} (F)=n_{\ssc K} (G)$; that is
 $$\sum_{j=1}^n n_j+\sum_{j=1}^q q_j= \sum_{j=1}^m m_j+\sum_{j=1}^p p_j.$$
 In other words,  $n_{\ssc K}(f)-p_{\ssc K}(f)=n_{\ssc K}(g)-p_{\ssc K}(g)$.
\end{proof}
Here is a variant of the preceding result. For a bounded  open set $G$ in $\C$, let
$MC(G)$ denote the set of functions continuous on $\ov G$ and meromorphic in $G^\circ$.
Note that, in general, $MC(G)$ cannot be represented as $M(K)$ for some compact space $K$.
For example, if $E\ss\D$ is a compact, nowhere dense  set having positive Lebesgue measure,
then  the planar integral $$f(z)=\dint_E \frac{1}{w-z} d\sigma_2(w)$$
belongs to $MC(\D\setminus E)$, but not to $M(\ov\D)$.

\begin{corollary}
For  a bounded open set $G\ss\C$,
suppose that  $f, g \in MC(G)$  are homotopic in $C(\partial G, \C^*)$.
Then $n_{\ssc G}(f)-p_{\ssc G}(f)=n_{\ssc G}(g)-p_{\ssc G}(g)$.
\end{corollary}

\begin{proof}
By assumption, $f$ and $g$ have no zeros and poles on $\partial G$. Hence, there are open
neighborhoods $U$ and $V$ of $\partial G$  with $\partial G\ss U\ss\ov U\ss V$
such that $f,g\in M(\ov G\setminus U)$ and $f$ and $g$   are homotopic in 
$C(\ov V\inter \ov G, \C^*)$ (for this latter point see \cite{moru}).  
 The assertion now follows from Theorem \ref{rouchhomo} if we set $K:=\ov G\setminus U$.
\end{proof}

A proof of the next Theorem is in  \cite[p. 97-101] {bu}.

\begin{theorem}[{\bf Eilenberg}] \label{eilenberg}
Let $K\ss\C$ be compact and 
for each bounded component $C$ of $\C\setminus K$, let $a_{\ssc C}\in C$. 
\begin{enumerate}
\item [(1)] Suppose that  $f:K\to\C\setminus\{0\}$ is continuous.  Then there exist
finitely many bounded components $C_j$ of $\C\setminus K$,  integers $s_j\in \Z$ $(j=1,\dots, n)$,
and $L\in \C(K)$ such that for all $z\in K$ 
$$f(z)=\prod_{j=1}^n (z-a_{\ssc C_j})^{s_j}\; e^{L(z)}.$$
\item [(2)] 
If for some $f\in C(K)$, $0$ belongs to the unbounded component of $\C\setminus f(K)$, then 
 $f$ has a continuous logarithm on $K$.
  \item[(3)] Suppose that  $C_1,\dots,C_n$ are distinct holes  for $K$ and that 
  for some $s_j\in\Z$, $(j=1,\dots, n)$, the function
  $$f(z)=\prod_{j=1}^n (z-a_{C_j})^{s_j},\sp (z\in K)$$
  has a continuous logarithm on $K$. Then $s_1=\dots=s_n=0$.
  \item [(4)] If $f:K\to \C$ is a homeomorphism, then the number of holes of $K$ and
  $f(K)$ coincide.
\end{enumerate}
\end{theorem}

\begin{proposition}\label{outerboundary}
Let $K\ss\C$ be   a compact set for which $\C\setminus K$ is connected and  let $G$ be a 
bounded component of $\C\setminus\partial K$. The following assertions hold:
\begin{enumerate}
\item [(1)] $G$ is simply connected.
\item[(2)] $\partial \ov G=\partial G$.
\item [(3)] ${\ov G\,}^\circ=G$.
\end{enumerate}
\end{proposition}
Item  (1) and the equivalence of (2) with (3) for  non-void open sets in general topological spaces 
are well known.
We include  a proof  of (1) and (2) for the reader's convenience.

 \begin{proof}
(1)  Let $\mathcal H:=\{G_n: n\in I\}$ be the set of holes of $\partial K$ and let 
$C:=(\C\setminus K)\union\partial K$. Let $n_0\in I$ be chosen so that $G=G_{n_0}$. 
 Note that $G_{n_0}$ is an open set and 
that for every $n$, $\partial G_n\ss \partial K\ss C$. Hence 
$$\C\setminus G_{n_0}= C\union \Union_{n\in I\atop n\not=n_0}G_n=
C \union \Union_{n\in I\atop n\not=n_0} \ov {G_n}.$$
Since $C=\ov{\C\setminus K}$, the assumption of the connectedness of $\C\setminus K$ implies 
that $C$ is connected.  Moreover,   $\ov {G_{n}}$ is connected for every $n$ and 
$\ov {G_{n}}\inter C\not=\emp$.
Hence  the union of all of these connected sets is connected;
that is $\C\setminus G_{n_0}$ is connected. Thus $G_{n_0}$ is a simply connected domain.

(2) First we note that for any set $M$ in any topological space, 
$\partial \ov M\ss \partial M$. 
The reverse inclusion now is a specific property of the set $G$.  So let $x\in \partial G$ and $U$
a neighborhood of $x$. Since the connectivity of $\C\setminus K$ implies that $\partial K=\partial\widehat K$
 we deduce from  $\partial G\ss \partial K$  
that $U$ meets the unbounded component of $\C\setminus K$.
 Since $\ov G=G\union \partial G\ss \widehat K=K$, $U$ cannot be entirely contained
 in $\ov G$. Hence $U$ meets the complement of $\ov G$ as well as $\ov G$. That is $x\in \partial \ov G$.
 We conclude that $\partial \ov G=\partial G$.
\end{proof}

Here is now the main result of this paper.
Recall that if $f\in P(K)$, then the Gelfand transform $f^*$ of $f$ is the unique continuous 
extension of $f$ to $\widehat K$ that is 
holomorphic in  $\widehat K^\circ$.
In particular, if $K\not=\widehat K$, then every function $f\in P(K)$ is holomorphic
in a neighborhood of each ``inner-boundary'' point $z_0\in \partial K\inter {\widehat K\,}^\circ$
(whenever they exist).

\begin{theorem}\label{maintheo}
Let $K\ss\C$ be compact.  Suppose that $f\in P(K)$ is injective. 
Then   $f^*$ is injective on $\widehat K$ if and only if 
 the outer boundary  $S_\infty$  of $K$
is mapped under $f$ onto the outer boundary  of $f(K)$. 
Moreover, in that case, $f^*(\widehat K)=\widehat{f(K)}$ and
each hole of $f(S_\infty)$ is the image under $f^*$
of a unique hole of $S_\infty$. 
\end{theorem}

Let us mention that Example  \ref{noinjext} provides  an 
injective function $f\in P(K)$ that does {\it not } map 
the outer boundary to the outer boundary.

\begin{proof}

(1)  Let $f^*$ be injective on $\widehat K$.
Note that $S_\infty=\partial \widehat K\ss \partial K$ and that the outer boundary of
$f(K)$ coincides with $\partial \widehat{f(K)}$.
 It remains to show that
\begin{equation}\label{pcimage}
\partial \widehat{f(K)}=\partial f^*(\widehat K)=f^*(\partial \widehat K).
\end{equation}
Here the second equality is satisfied due to the assumption that $f^*$ is a homeomorphism
between $\widehat K$ and $f^*(\widehat K)$.
Now $\widehat K$ is  polynomially convex.  
Hence,  by Theorem \ref{eilenberg} (4), $f^*(\widehat K)$ has no holes. 
Consequently,  $\partial f^*(\widehat K)$ is the outer boundary of $f^*(\widehat K)$  and the polynomial convexity of $f^*(\widehat K)$ implies that 
$$
\widehat{f(K)}\ss f^*(\widehat K).
$$
But we also have the reverse inclusion.
 In fact, let  $\check w=f^*(\check z)\in f^*(\widehat K)$,
where $\check z\in \widehat K$. 
Since $p\circ f\in P(K)$  for every polynomial $p\in \C[z]$, we conclude from 
$\max_K |h|=\max_{\widehat K}|h^*|$ for every $h\in P(K)$, that 
$$|(p\circ f)^*(\check z)|\leq \max_{z\in K} |(p\circ f)(z)|.$$
Hence
$$ |p(\check w)|\leq \max\{|p(y)|: y\in f(K)\}.$$
In other words, $\check w\in \widehat{f(K)}$. This implies that 
\begin{equation}\label{hullinclu}
f^*(\widehat K)\ss \widehat{f(K)}.
\end{equation}
(Note that (\ref{hullinclu}) holds
independently of $f^*$ being injective or not.)
Thus 
\begin{equation}\label{hullsequal}
f^*(\widehat K)= \widehat{f(K)},
\end{equation}
 and therefore  $\partial \widehat{f(K)}=\partial f^*(\widehat K),$
 which establishes (\ref{pcimage}).\\

(2) Next we prove the converse.
We may  assume that $K$ is not polynomially convex, otherwise there is nothing to show. In particular, ${\widehat K\,}^\circ\not=\emp$. 
So  suppose that $\partial \widehat{f(K)}=f(\partial \widehat K)$.\\

{\bf Step 1} We show that  $f^*|_G$ is injective for every hole $G$ of $\partial \widehat K$.

Let $M:=f(\partial G)$ and $S:=\widehat{f(K)}$.   Then $\partial S$ is the outer boundary
of $f(K)$, and
$$M=f(\partial G)\ss f(\partial \widehat K)=\partial \widehat{f(K)}=\partial S.$$

 Let $a$ belong to the unbounded  component,  $\Omega_\infty$, of $\C\setminus M$.
 Then $0$  belongs to the unbounded component of $\C\setminus (f-a)(\partial G)$. 
 By Theorem \ref{eilenberg}(2), 
$f(z) -a =e^{L(z)}$ for some $L\in C(\partial G,\C)$. Hence $f-a$ is homotopic
in $C(\partial G,\C^*)$ to 1. Since $\partial G=\partial \ov G$,  (Proposition \ref{outerboundary})
we conclude from Theorem  \ref{rouchhomo} that  $f^*-a$ has no zeros in 
${\ov G\,}^\circ=G$. Hence 
\begin{equation}\label{inclu1}
f^*(G)\ss \widehat M. 
\end{equation}

Next, we claim that $f^*(G)\inter \partial S=\emp$. To see this, let us suppose that there exists
$z\in G$ with $f^*(z)\in \partial S$. Since $f^*$ is holomorphic in $G$ (and
due to the injectivity on the boundary,  not constant on $G$), we conclude that
 $f^*$ is an open map on $G$. 
 Hence a whole  disk  $D(f^*(z),\e)$ belongs to $f^*(G)$. Thus $f^*(G)$ meets the unbounded
component  $C_\infty$, of $\C\setminus S$ (note that $S$ is polynomially convex). This is 
 a contradiction because $ C_\infty\ss \Omega_\infty$ and no point in $\Omega_\infty$ belongs
 to $f^*(G)$, as was shown above. Consequently, $f^*(G)\inter \partial S=\emp$.

Because  $\widehat M=\widehat{f(\partial G)}\ss \widehat {f(K)} =S$, we then  
conclude from (\ref{inclu1}) that 
$f^*(G)\ss \widehat M\setminus \partial S\ss S\setminus \partial S$. 
But $S^\circ\not=\emp$, since the open set $f^*(G)$ is contained in 
$f^*(\widehat K)\buildrel\ss_{}^{(\ref{hullinclu})}  \widehat{f(K)}= S$.
Hence $S\setminus \partial S$ is a non-void  open set. Because $\C\setminus S$ is connected,
  $S\setminus \partial S$ consists of the union of all holes of $\partial S$. 
Thus the connected set $f^*(G)$ is contained in a unique hole, $H$, of $\partial S$.

Next we show that every point in $H$ is taken once by $f^*$ on $G$.  
For technical reasons, we suppose that  $0\in G$ (otherwise we use an appropriate translation).

Fix $b\in H$. Let $g:\partial S\to S_\infty\ss K$ be the restriction to $\partial S$
of the  inverse of $f$ (here  we have used  the hypothesis that $f$ maps the outer
boundary $S_\infty$ of $K$ onto the outer boundary $S$ of $f(K)$).
Note that $g$ does not take the value $0$ because, by assumption,  
$0 \in \C\setminus \partial\widehat K$.
By Theorem \ref{eilenberg}(4),   $\partial S$ and $S_\infty$ have the same number of holes.
Let $\mathcal H:=\{H_j: j\in I\}$ be the set of holes of $\partial S$. We may assume that $H_1=H$.
Fix in each hole $H_j$ of $\partial S$ a point $b_j$, $(j\in I\ss\N^*)$, where  we take $b_1=b$.
 By Eilenberg's Theorem \ref{eilenberg}, there exists $n\in \N$, 
  $L\in C(\partial S,\C)$ and $s_j\in \Z$ such that
$$\mbox{$g(w)=\prod_{j=1}^n(w-b_j)^{s_j}e^{L(w)}$ for every $w\in \partial S$.}$$
If $z:=g(w)$ (or equivalently $w=f (z)$), then $z\in \partial \widehat K= S_\infty\ss \partial K$ and
\begin{equation}\label{specialhomo}
\mbox{$z= \prod_{j=1}^n(f(z)-b_j)^{s_j} e^{L(f(z))}$ for these  $z $.}
\end{equation}
 In particular
\begin{equation}
H(z,t):= \prod_{j=1}^n(f(z)-b_j)^{s_j} e^{tL(f(z))}
\end{equation}
is a homotopy in $C(\partial G,\C^*)$ between
 the function $\prod_{j=1}^n(f(z)-b_j)^{s_j}$ and the identity function $z$.
Now, for $z\in \widehat K$,  $$\psi(z):=\prod_{j=1}^n(f^*(z)-b_j)^{s_j}$$
is  a meromorphic function in $M(\widehat K)$.  Also, $\partial G=\partial\ov  G$ and
${\ov G\,}^\circ=G$ (Proposition \ref{outerboundary}).
Hence, by Theorem \ref{rouchhomo},
$n_G(\psi)-p_G(\psi)=1$. Since $f^*(G)\ss H_1$,  $\psi|_G(z)=(f^*(z)-b_1)^{s_1} R(z)$,
where $R$ is zero-free and holomorphic on $G$. We conclude that $s_1=1$ and $f^*(z_1)=b_1$ 
for a unique $z_1\in G$.
Hence $f^*$ is a bijection of $G$ onto $H_1$. Since $f(\partial G)\ss \partial S$, $f^*$  actually is
a bijection from $\ov G$ onto $\ov H_1$. \\

{\bf Step 2} We claim that $f^*$ is injective on $\widehat K$. 
It only remains to show that  $f^*(G)\inter f^*(C)=\emp$ whenever $G$ and $C$ are two different holes
of $S_\infty=\partial\widehat K$. To see this, suppose that $f^*(G)\inter f^*(C)\not=\emp$.
Since the images of $G$ and $C$ under $f^*$ are holes of  $\partial S$, we conclude that
$f^*(C)=f^*(G)=H_1$.  Moreover, 
$$f^*(\partial G)=\partial f^*(G)=\partial f^*(C)=f^*(\partial C).$$
The injectivity of $f$ on $\partial K$ and the fact that $\partial C\union \partial G\ss\partial K$
now imply that $\partial G=\partial C$. Moreover, $\partial \ov C=\partial C$.
Since $0\in G\not=C$, we conclude from (\ref{specialhomo}) and Theorem \ref{rouchhomo} that
$n_C(\psi)-p_C(\psi)=0$.  On the other hand,  since $f^*(C)\ss H_1$,  $\psi|_C(z)=(f^*(z)-b_1)^{s_1} R(z)$,
where $R$ is zero-free and holomorphic on $C$.  Now $s_1=1$ implies
 that $p_C(\psi)=0$. Hence $n_C(\psi)=0$, too. 
This is  contradiction, though,  because $f^*(C)=H_1$ and $b_1\in H_1$. 
Thus we have shown that $f^*$ is a bijection of $\widehat K$ onto $f^*(\widehat K)$.
 
(3) If  $f^*$ is a homeomorphism of $\widehat K$ onto its image $f^*(\widehat K)$, then
we have already shown   that  $f^*(\widehat K)=\widehat{f(K)}$ (see \ref{hullsequal}).
Hence,  we conclude from the
preceding  paragraphs (applied to $(f^*)^{-1}$) that each hole $H$ 
of $\partial\widehat{f(K)}=f(S_\infty)$ 
writes as $H=f^*(G)$ for some
uniquely determined hole $G$ of $S_\infty=\partial\widehat K$.
\end{proof}

A natural question  is whether a compactum  $K$ with  a single hole has the so-called
{\it extension property}, that is if $f\in P(K)$ is injective, then $f^*$ is injective on $\widehat K$.
A slight modification of Example \ref{noinjext} shows that this is  not true, either:

\begin{example}\label{exo2}
Let
 $$K_1=\{z\in\C: |z+1|\leq 1\}\union \{z\in\C: |z-2|=2\}$$ 
 (see figure \ref{inji2}).
 \begin{figure}[h] 
   \scalebox{.35} {\includegraphics{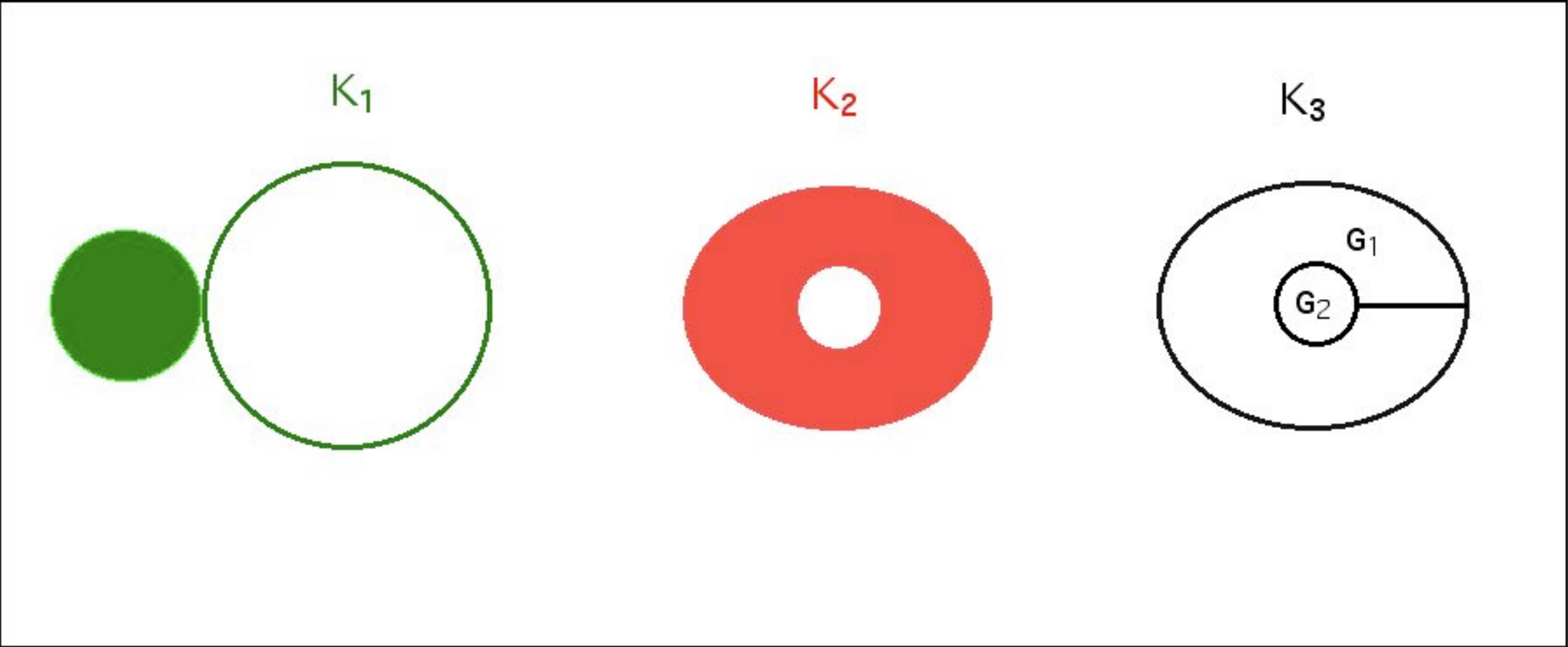}} 
\caption{\label{inji2} Regular and non-regular holes}
\end{figure}

Then the function $f(z)=-z$ for $|z+1|\leq 1$ and $f(z)=z$ for $|z-2|=2$
belongs to $P(K_1)$, but of course, by the same reasoning as in Example \ref{noinjext}
$f^*$ is not injective on $\widehat K_1$. 
\end{example}

So let us modify the question a little bit:
let $G$ be a hole of $K$ and suppose that $f\in P(K)$ is injective. Is $f^*|_G$ injective?
See figure \ref{inji2} for several examples.
In the following, a positive answer  will be given for a special class of holes.

\begin{definition}
Let $K\ss\C$ be compact and $G$  a hole of $K$. Then $G$ is called a {\rm regular hole}
if $G$ is the only hole of its boundary $\partial G$; that is if 
$\widehat{\partial G}=G\union \partial G=\ov G$.
\end{definition} 

In figure \ref{inji2}, the holes  of $K_1$  and $K_2$ are regular as well as the hole
$G_2$ of $K_3$,  but  $G_1$ is not regular.  A more interesting class of non-regular holes is
provided by Example \ref{1-2-hole}. It has the additional property that $G_1$ is a component
of the interior of a {\it polynomially convex} set $K$.

\begin{example}\label{1-2-hole}
There is a compact set $K\ss\C$  with connected complement such that some  hole $G_1$ of $\partial K$ 
has the property that $G_1$ is not the unique hole of $\partial G_1$.
 \end{example}

 \begin{figure}[h!] 
   \hspace{2cm}
   \scalebox{.40} {\includegraphics{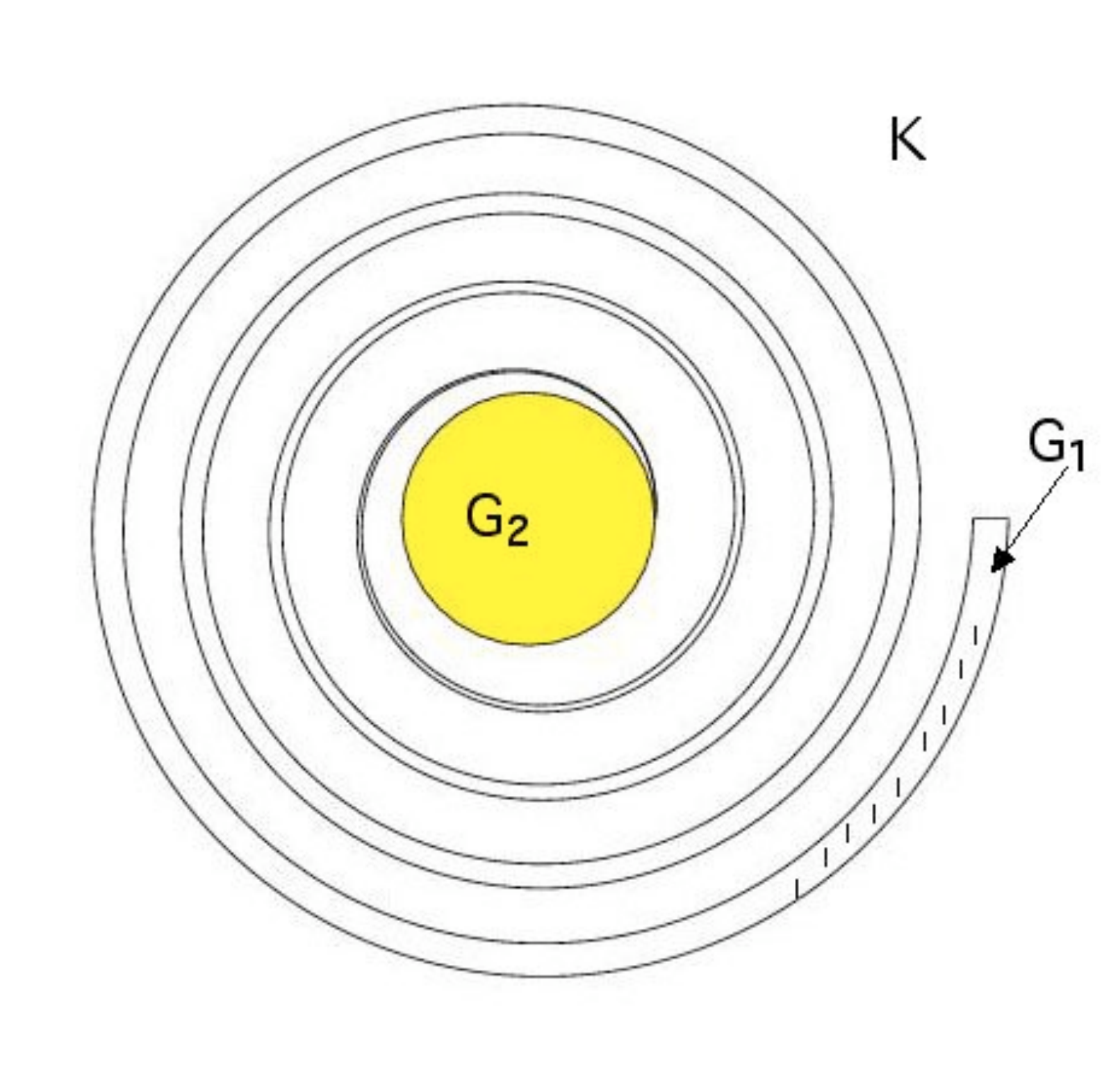}} 
\caption{\label{spiral3} A p.c. compactum with  a boundary
 hole whose boundary induces {\it two} holes}
\end{figure}

\begin{proof}
Let $K$ be the union  of the closed unit disk with a ``thick" spiral $S$ surrounding  the unit circle infinitely
often and clustering  exactly at  every point of $\T$ (see figure \ref{spiral3}). Then $\C\setminus K$ 
is connected, and the holes of $\partial K$ are the components of $K^\circ$; 
these are  the interior $G_1$  of the spiral $S$ and the open unit disk, denoted here by $G_2$. 
Then $\partial G_1=\partial K$; hence $G_1$ and $G_2$ are the holes of  the boundary
of the hole $G_1$ of $\partial K$. 
\end{proof}

This example also shows that the closure $\ov G_1$ of the  component $G_1$ of  the polynomial convex set $K$, may have  a disconnected complement, although $G_1$ itself is simply connected.

 It actually can happen  that two, or even infinitely many, holes of a compactum may
have the {\it same} boundary. These sets are known under the name ``lakes of Wada'', 
 first discovered by L.E.J. Brouwer \cite{br},  see also \cite[p.~138] {gelo}.

\begin{lemma}\label{bounds}
Let $G\ss\C$ be a bounded domain with ${\ov G\,}^\circ =G$ and
$$\widehat{\partial G}=G\union \partial G\; \footnote{ In other words,
$G$ is the only hole of $\partial G$.}.
$$
If $f:\partial G\to \C$ is a continuous injective map, then 
 $f(\partial G)$  is the boundary of a bounded domain $H$ with ${\ov H\,}^\circ =H$ and
 $$\widehat{\partial H}=H\union \partial H.$$
\end{lemma}

\begin{proof}
By Theorem  \ref{eilenberg}(4), 
$E:=f(\partial G)$ has a single hole, too. Let us denote this hole by $H$. Since $\partial H\ss \partial E$, we have 
\begin{equation}\label{ehe}
\widehat E=E\union H=E\union \ov H.
\end{equation}.
 Note  that $\partial \ov H\ss \partial H\ss E$.
We claim that $\partial \ov H=E$. Suppose, to the contrary, that $S:=\partial\ov H\subset E$, 
the inclusion being strict. Let $F:=f^{-1}(S)$. Then $F$ is a proper, closed  subset of $\partial G$.
Since $\partial G\setminus F$ is relatively open in the closed set  $\partial G$, 
there is  $\xi\in \partial G$
and a disk $D=D(\xi,\e)$  such that $D\inter F=\emp$. Let
$$U:=G\union (\C\setminus \ov G) \union D.$$ 
By hypothesis,  $\widehat{\partial G}=\ov G$. Hence $\C\setminus \ov G$ is connected 
(because it coincides with the unbounded complementary component of the polynomially convex set $\widehat{\partial G}$). 

Because the hypothesis ${\ov G\,}^\circ=G$ implies that $\partial G=\partial\ov G$, 
we conclude that
$D$ meets $G$ as well as  $\C\setminus \ov G$. Hence,
  $U$ is an unbounded open connected set
contained in the open set $\C\setminus F$.  Thus $U$ is contained in
 the unbounded component of $\C\setminus F$.
Since the remaining part $(\C\setminus F) \setminus U\ss \partial G\setminus F$
of $\C\setminus F$ is small in the sense that it does not contain interior points,
$\C\setminus F$ does not have a bounded component.  In other words, 
$F$ has no holes. This is a contradiction, because $F$ has the same number of holes as  $S$;
 that is at least one hole.
Thus we have shown that $\partial \ov H=\partial H=E$. The identity 
$\widehat E= E\union H$ (see \zit{ehe}) now implies that
 $\widehat {\partial H} = \partial H\union H$.
\end{proof}

\begin{theorem}\label{dpforpk}
Let $K\ss\C$ be compact and suppose that $f\in P(K)$  is injective. 
If  $G$ is  a hole of   the outer boundary $S_\infty$ of $K$,  
 then the restriction $f^*|_{\ov G}$ of the Gelfand transform $f^*$ of $f$ to $\ov G$
 is injective whenever $G$ is the only hole of $\partial G$.
\end{theorem}

Example  \ref{1-2-hole} shows that the strange condition
``whenever $G$ is the only hole of $\partial G$'' is not always satisfied. 

\begin{proof}
Because $G$  is  the only hole of $\partial G$, we have
$\widehat{\partial G}=G\union \partial G=\ov G$.  Thus  $M:=\ov G$   is polynomially convex.
Hence,  the outer boundary of $M$ coincides with $\partial M=\partial\ov G$.
Moreover, 
since  $G$ is a hole of the boundary $S_\infty$ of the polynomially convex set
$ \widehat {S_\infty}$, we obtain from  Proposition \ref{outerboundary} 
that $\partial G=\partial \ov G$ and that ${\ov G\,}^\circ =G$.  

Since $\partial M$ has a single hole, namely, $G={\ov G\,}^\circ$,  and since $f$ is injective on 
$\partial M$,  $E:=f(\partial M)$ has a single  hole, too.
Let $H$ be that hole.    By Lemma \ref{bounds},  $\widehat{\partial H}=H\union \partial H$ and
$\partial H=\partial \ov H=E$. 
We conclude that $f$ maps the outer boundary $\partial M$ of $M$  onto the outer boundary 
$E$ of $\widehat{f(\partial M)}$.
By Theorem \ref{maintheo}, $f^*$ is injective on $M=\ov G$. 
\end{proof}

Example \ref{noinjext} shows that, in general,  $f^*$ is not injective on the union of two bounded  components $G_j$ of $\C\setminus S_\infty$.   However, we don't know whether $f^*|_G$
is injective in case $G$ is not a regular hole of $S_\infty$. 

\begin{corollary}\label{3extensions}
Let $X\ss \C$ be compact  and $H$ a hole of $X$.   Suppose that  $f\in P(X)$ is injective. 
Under each of the following conditions $f^*$ is injective on $\ov H$:
\begin{enumerate}
\item [(1)]  $(\partial\ov H, f)$ satisfies the condition of Theorem \ref{maintheo} 
with $K=\partial\ov H$.
\item [(2)]  $H$ is contained in a hole $G$ of the outer boundary of $X$  which has the property
that $G$ is the only hole of $\partial G$.
\item [(3)] $H$ is  a regular hole of $X$.
\end{enumerate}
\end{corollary}
\begin{proof}
(1) and (2) are clear.

(3) Let $M=\ov H$.  By hypothesis,  $\widehat {\partial H}=H\union \partial H$. Thus
$M$ is polynomially convex.  Since $H\ss {\ov H}^\circ\ss \ov H$, we conclude from the connectedness of $H$ that $G:={\ov H\,}^\circ$ is connected.  Hence $G$ is the only
hole of $\partial \ov H$. Since $\partial\ov H$ is the outer boundary of $\ov H$,  it follows that  $\partial \ov H=\partial G$ and  $\ov G=\ov H$. In particular, 
$\widehat{\partial G}=\partial G\union G$.
By Theorem  \ref{dpforpk}, $f^*|_{\ov G}$ is injective. 
\end{proof}
\begin{corollary}
Let $K\ss\C$ be  compact. Suppose that $\C\setminus K$ and $K^\circ$ are connected. 
Then $\partial K$ has the extension property.
 \end{corollary}
 \begin{proof}
 If $K^\circ=\emp$, then the polynomial convexity of $K$ implies that $\widehat K=K=\partial K$.
 Hence the assertion is trivial. So let us assume that $K^\circ\not=\emp$.
 Let $M=\ov{K^\circ}$.  We claim that   $M$ is polynomially convex. In fact,
 $$\ov{K^\circ}\ss \widehat{\ov{K^\circ}}\ss \widehat K =K.$$
If $ \ov{K^\circ}$ would be a strict subset of  $\widehat{\ov{K^\circ}}$, then 
 $\ov{K^\circ}$ would have a hole $H$. Hence
 $$\ov{K^\circ}\union H\ss \widehat{\ov {K^\circ}}\ss K.$$
 Consequently, $K^\circ \union H\ss K^\circ$; this is an obvious contradiction. We conclude that
 $$\widehat{\partial K^\circ}=\widehat{\ov{K^\circ}}=\ov{K^\circ}=K^\circ\union \partial K^\circ.$$
 Thus  $K^\circ$
 is  a regular hole for $\partial M$.  The conclusion now  follows from Corollary \ref{3extensions}.
  \end{proof}
  
 Examples \ref{exo2} and   \ref{noinjext} (this latter for the full disks) 
  show that neither of the  conditions $\C\setminus K$ connected or 
  $K^\circ$ connected implies 
 that $\partial K$ has the extension property.

 Now let $K\ss\C$ be  a compact set for which $\partial K$ has the extension property
  (for $P(K)$-functions).
 If $f\in R(K)$ is injective on $\partial K$, does this imply that $f$ is injective on $K$?
 The following example shows that this is not necessarily the case:

\begin{example}

Let $K=\{z\in \C: r\leq |z|\leq R\}$ where $0<r<1<R$ and $rR\not=1$.  Then the function  $f$, given by 
$f(z)=z+\frac{1}{z}$ belongs to $R(K)$, is injective on $\partial K$, but not on $K$. In fact, 
$f(z)=f(w)$ implies that $z-w = (w-z)/ zw$. Since  on $\partial K$,  $zw\not=1$, we have $z=w$.
On the other hand, $f(i)=f(-i)=0$. 
\end{example}

Finally, we want to present the following poblem: suppose that  $f\in C(\partial K,\C)$
is injective. Under which conditions $f$ admits a continuous injective extension to $K$
or even $\C$? Note that if $K$ is the closure of a Jordan domain, then the Schoenflies theorem
guarantees the existence of a homeomorphism of $\C$ extending $f$.

\section{Continuous logarithms on compact sets containing the origin on their boundary}\label{loags}

Eilenberg's Theorem \ref{eilenberg}(2)  shows that if $0$ belongs to the
unbounded complementary component of  a compact set $K$ in $\C$,
 then there exists a continuous branch
of the logarithm of $z$ on $K$. On the other hand,  by \ref{eilenberg}(3),  if $0$
belongs to a bounded complementary component of $K$, 
then there does not exist a continuous function $h$
on $K$ such that $e^{h(z)}=z$ for every $z\in K$. We will investigate now the case
when $0$ belongs to the boundary of $K$.  Does there exist a continuous branch  of $\log z$
on $K\setminus \{0\}$?   The answer is ``not necessarily" \footnote{This refutes statements and invalidates the associated ``proofs'' in \cite[p. 62]{st} and its verbatim copy in \cite[p. 348]{kpc}}.

\begin{proposition}\label{pclog}
 There exists a compact set $K$ in $\C$ with    $0\in \partial K$ and connected complement such that
no continuous branch of $\log z$ can be defined on $K\setminus \{0\}$.
\end{proposition}

\begin{proof}

Let $E$ be the disk $\{z\in\C: |z+1|\leq 1\}$ and $S$ a spiral
starting at 1 and surrounding $E$ infinitely often and clustering at every
point on the boundary of $E$; for example one may describe $S$ as the half-open curve 
$$z(t)= -1+\left(1+ \frac{1}{1+t}\right)e^{it},\sp 0\leq t<\infty.$$

Let $K=E\union S$. Then $K$ is compact and polynomially convex.  Note also that
$\ov S\inter E=\partial E$.
Moreover, $0$ is a boundary point of $K$. We show that there does not
exist a continuous branch of $\log z$ on $K\setminus \{0\}$.\bigskip

 \begin{figure}[h] 
   \hspace{2cm}
   \scalebox{.40} {\includegraphics{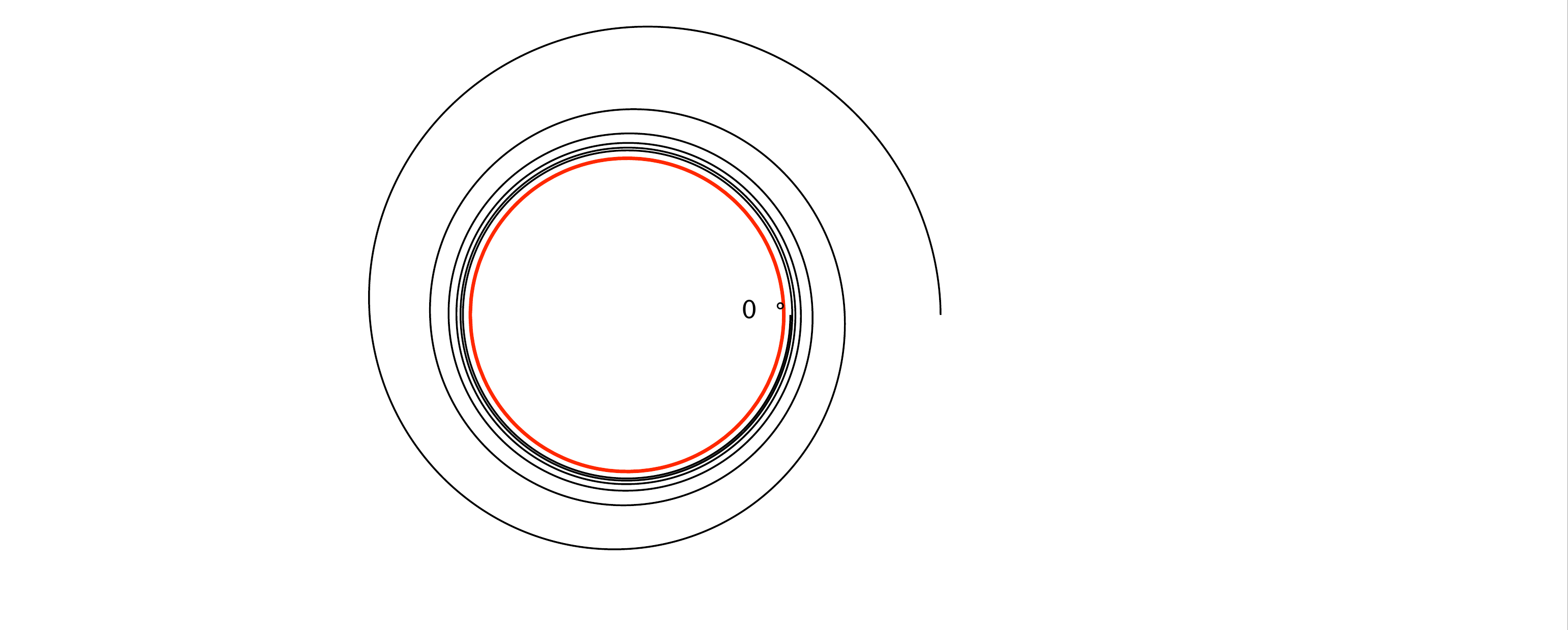}} 
\caption{\label{spiral1} A spiral clustering at a circle}
\end{figure}

In fact, since $S$ is a connected set surrounding $0$ infinitely often,
any continuous determination of the argument of $z$ when $z$ runs through
the spiral $S$ has to be unbounded. This can be seen by geometric intuition or by the following
analytic argument:

If we look at 
$w(t):=\exp(-it)z(t)=1+1/(1+t)-\exp(-it)$, $0\leq t<\infty$,  then ${\rm Re}\, w(t)\geq 1/(1+t)>0$.
Hence $w(t)$ belongs to the right halfplane. Let $L(z)=\log z$ be the principal branch
of the logarithm on the right half-plane and set $h(t):=L(w(t))$. Then
$$\exp(-it)z(t)=\exp(h(t)).$$
Therefore, $z(t)=\exp(it +h(t))$. Because  $|{\rm Im}\; h(t)| \leq \pi/2$,
$$\arg  z(t)={\rm Im}\,  (it+h(t))$$  behaves as $t$ for large $t$. 
Thus the imaginary part of $\log z$ is unbounded,  for $z\in S$. 

Since the spiral $S$ clusters at every point of the circle $C:=\{|z+1|=1\}$ and $C\ss \ov S\ss K$,
$\log z$ cannot be continuous on $K\setminus\{0\}$. 
\end{proof}

Next we give a sufficient condition   for the existence of such logarithms.

\begin{definition}
 A boundary point $z_0$ of a compact set $K$ is said to be {\it accessible},
if there is a Jordan arc   $\gamma: \;]0,1[ \to \C\setminus K$ comming from 
infinity and ending at $z_0$ (that is $\lim_{t\to 0}\gamma(t)=\infty$ and 
$\lim_{t\to 1}\gamma(t)=z_0$).
\end{definition}
We note that it is well known that the set of accessible boundary points for $K$
 is dense in the boundary $\partial K$ of $K$.

\begin{theorem}
 Let $K$ be a compact set in $\C$ and suppose that $0\in \partial K$.
If $0$ is an accessible boundary point, then there is a continuous branch
of $\log z$ on $K\setminus \{0\}$.
\end{theorem}
\begin{proof}
 Let $J=\gamma (]0,1[)$ be a Jordan arc in the complement of $K$,
 joining $\infty$ with $0$; in particular,  $\lim_{t\to 0}\gamma(t)=\infty$ and 
$\lim_{t\to 1}\gamma(t)=0$. Note that $\ov J=J\union \{0\}$.
 Then $\Omega:=\C\setminus \ov J$ is a simply connected domain in $\C$ with
$0\notin \Omega$. Hence
there is a holomorphic branch of $\log z$ in $\Omega$.
Because $K\setminus \{0\}\ss \Omega$, we have obtained the desired logarithm.
\end{proof}

For example if $K$ is the union of $\{0\}$ with the spiral 
parametrized by $$z(t)=\left\{\frac{1}{1+t}e^{it}: 0\leq t<\infty\right\},$$
then $0$ is an accessible boundary point of $K=\partial K$ and $\log z(t)=it -\log(1+t)$
is a continuous branch of the logarithm on $K\setminus\{0\}$.

 \begin{figure}[h] 
   \hspace{1cm}
   \scalebox{.30} {\includegraphics{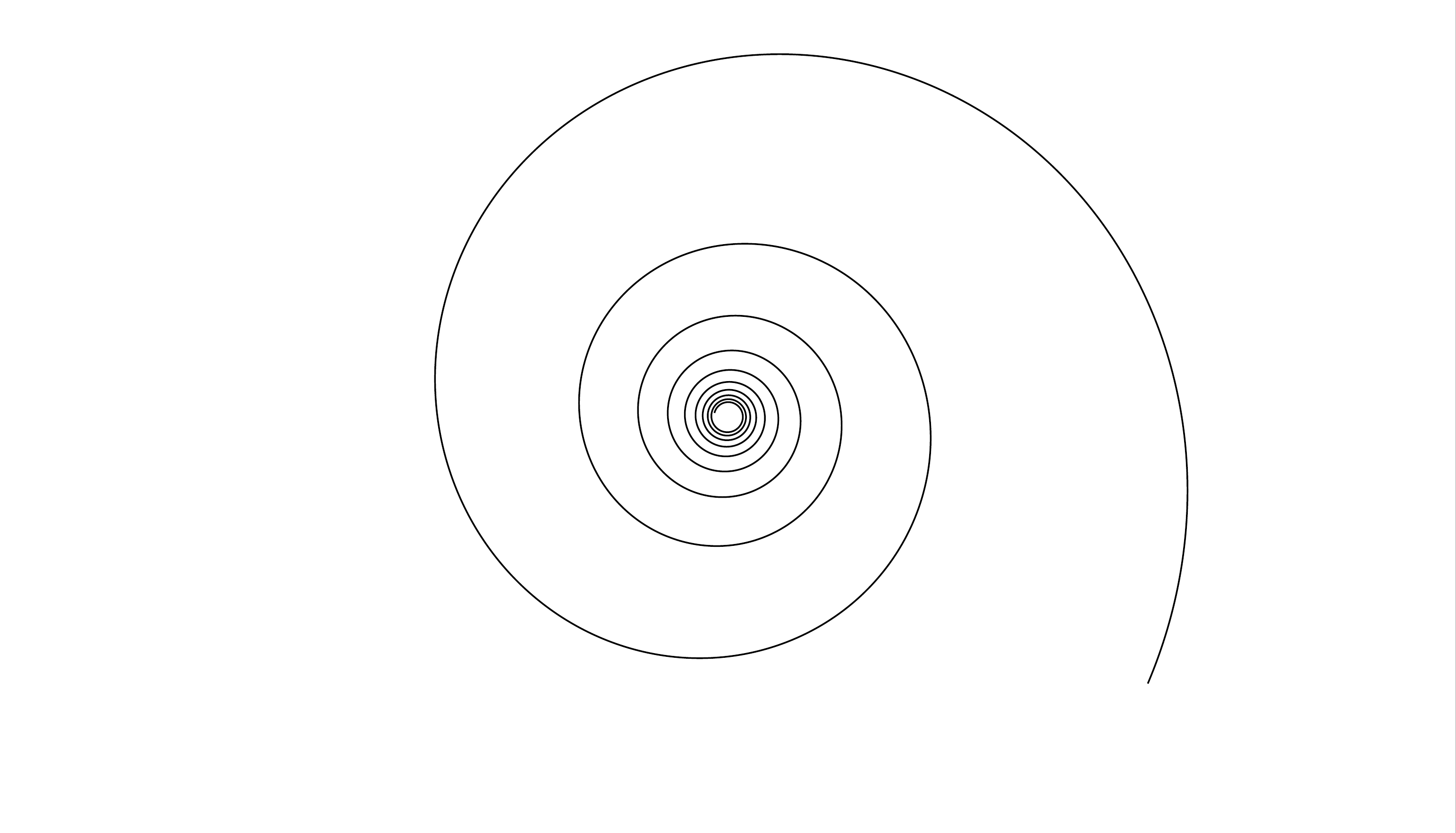}} 
\caption{\label{spiral2} A spiral ending at the origin}
\end{figure}

It is not known at present, whether accessibility characterizes the compact sets under discussion.
\bigskip

\bigskip

{\bf Acknowledgements}

I thank Robert Burckel for his comments on section \ref{loags}, 
Lee Stout for some helpful comments on a preliminary 
version of this work and for reference \cite{br},
 Rainer Br\"uck for reference  \cite{pom} in connection with Theorem
\ref{dbp} and J\'er\^ome No\"el for drawing figure \ref{spiral3}.
I also thank Rudolf Rupp for his contribution to Theorems \ref{ck} and \ref{ak}.

 \end{document}